\documentclass{amsart}
\usepackage{graphics}
\usepackage{amsfonts,amsmath}
\begin{document}

 \newtheorem{thm}{Theorem}[section]
 \newtheorem{cor}[thm]{Corollary}
 \newtheorem{lem}[thm]{Lemma}{\rm}
 \newtheorem{prop}[thm]{Proposition}

 \newtheorem{defn}[thm]{Definition}{\rm}
 \newtheorem{assumption}[thm]{Assumption}
 \newtheorem{rem}[thm]{Remark}
 \newtheorem{ex}{Example}
\numberwithin{equation}{section}
\def\la{\langle}
\def\ra{\rangle}
\def\glexe{\leq_{gl}\,}
\def\glex{<_{gl}\,}
\def\e{{\rm e}}

\def\x{\mathbf{x}}
\def\P{\mathbb{P}}
\def\h{\mathbf{h}}
\def\by{\mathbf{y}}
\def\bz{\mathbf{z}}
\def\F{\mathcal{F}}
\def\R{\mathbb{R}}
\def\T{\mathbf{T}}
\def\N{\mathbb{N}}
\def\D{\mathbf{D}}
\def\V{\mathbf{V}}
\def\U{\mathbf{U}}
\def\K{\mathbf{K}}
\def\Q{\mathbf{Q}}
\def\M{\mathbf{M}}
\def\oM{\overline{\mathbf{M}}}
\def\O{\mathbf{O}}
\def\C{\mathbb{C}}
\def\P{\mathbb{P}}
\def\Z{\mathbb{Z}}
\def\H{\mathcal{H}}
\def\A{\mathbf{A}}
\def\V{\mathbf{V}}
\def\AA{\overline{\mathbf{A}}}
\def\B{\mathbf{B}}
\def\c{\mathbf{C}}
\def\L{\mathcal{L}}
\def\bS{\mathbf{S}}
\def\H{\mathcal{H}}
\def\I{\mathbf{I}}
\def\Y{\mathbf{Y}}
\def\X{\mathbf{X}}
\def\G{\mathbf{G}}
\def\f{\mathbf{f}}
\def\z{\mathbf{z}}
\def\v{\mathbf{v}}
\def\y{\mathbf{y}}
\def\d{\hat{d}}
\def\bx{\mathbf{x}}
\def\bI{\mathbf{I}}
\def\y{\mathbf{y}}
\def\g{\mathbf{g}}
\def\w{\mathbf{w}}
\def\b{\mathbf{b}}
\def\a{\mathbf{a}}
\def\u{\mathbf{u}}
\def\q{\mathbf{q}}
\def\s{\mathcal{S}}
\def\cc{\mathcal{C}}
\def\co{{\rm co}\,}
\def\tg{\tilde{g}}
\def\tx{\tilde{\x}}
\def\tg{\tilde{g}}
\def\tA{\tilde{\A}}

\def\supmu{{\rm supp}\,\mu}
\def\supp{{\rm supp}\,}
\def\cd{\mathcal{C}_d}
\def\cok{\mathcal{C}_{\K}}
\def\cop{COP}

\title
{Existence of Gaussian cubature formulas}
\author{Jean B. Lasserre}
\address{LAAS-CNRS and Institute of Mathematics\\
University of Toulouse\\
LAAS, 7 avenue du Colonel Roche\\
31077 Toulouse C\'edex 4, France\\
Tel: +33561336415}
\email{lasserre@laas.fr}

\date{}
\begin{abstract}
We provide a necessary and sufficient condition for existence of Gaussian cubature formulas.
It consists of checking whether some overdetermined linear system has a solution
and so complements Mysovskikh's theorem which requires computing
common zeros of orthonormal polynomials. 
Moreover, the size of the linear system shows that existence of a cubature formula imposes 
severe restrictions on the associated linear functional. For fixed precision (or degree),
the larger the number of variables the worse it gets. And for fixed 
number of variables, the larger the precision the worse it gets. Finally,
we also provide an interpretation of the necessary and sufficient condition
in terms of existence of a polynomial with very specific properties.
\end{abstract}



\maketitle

\section{Introduction}

In addition of being of self-interest, Gaussian quadrature formulas for one-dimensional integrals are particularly important because
among all other quadrature formulas with same number of nodes, 
they have maximum precision. Indeed, when supported on $p$ nodes,
they are exact for all polynomials of degree at most $2p-1$. Moreover,
the nodes are exactly the zeros of the orthogonal polynomial of degree $p$ and
existence is guaranteed whenever the 
moment matrix of the associated linear functional is positive definite. 

When its multi-dimensional analogue exists (then called Gaussian {\it cubature} formula), it shares those two important properties
of support and precision. That is, the nodes of a cubature formula of degree (or precision) $2p-1$
are the common zeros of all orthonormal polynomials of degree $p$.
However, and in contrast to the one-dimensional case,
its existence is not guaranteed even if the moment matrix of
its associated linear functional is positive definite. 

To the best of our knowledge, the only necessary and sufficient condition available is
\cite[Theorem 3.7.4]{dunkl} due to Mysovskikh, which states that an $n$-dimensional Gaussian cubature formula
of degree $2m-1$ exists if and only if the orhonormal polynomials of degree $m$ have exactly 
$s_{m-1}:={n+m-1\choose n}$ common zeros. So this condition requires computing those zeros, or, equivalently,
all joint eigenvalues of $n$ associated Jacobi matrices; see e.g. \cite[Theorem 3.6.2]{dunkl}.
Equivalently, this amounts to check whether some $n$ {\it multiplication} matrices of size
${n+m-2\choose m-1}$ commute pairwise; see \cite[Theorem 3.6.5]{dunkl}.

On the other hand, in Cools et al.  \cite{cools}
one may find several lower bounds on the number of nodes required for a given specific precision;
see e.g. \cite[Theorem 9, 10, 11,12, 13]{cools}. For more details on orthogonal polynomials and
cubature formulas, the interested reader is referred to e.g. the excellent book of Dunkl and Xu \cite{dunkl} and the survey by
Cools et al. \cite{cools}.

{\bf Contribution.} This  paper is concerned with existence of Gaussian cubature formulas associated with
a Borel measure $\mu$ on $\R^n$, with all moments finite.
Our contribution is to provide a necessary and sufficient condition for existence, different in spirit from
(and complementary to) Mysovskikh's theorem which requires solving polynomial equations,
or, equivalently, computing all joint eigenvalues of  $n$ Jacobi matrices.
Namely, we prove that a Gaussian cubature formula or precision $2m-1$ exists if and only if a certain 
overdetermined linear system
has a solution. The coefficient matrix of this linear system comes from expressing the product 
$P_\alpha P_\beta$ of any two orthonormal polynomials  $P_\alpha,P_\beta$ of degree $m$, in the basis of
orthonormal polynomials $(P_\theta)$, of degree at most $2m$.
In fact, in this expansion, only the constant coefficient and the coefficients of the 
degree-$2m$ orthonormal polynomials matter. This linear system is always
overdetermined and the larger $n$ for fixed $m$ (resp. the larger $m$ for fixed $n$) the worse it gets.
This shows that existence of a Gaussian cubature formula imposes severe restrictions on the associated linear functional. 

Finally, we also provide an interpretation of the necessary and sufficient condition, namely that
there exists a polynomial $Q$, which is a linear combination of orthonormal polynomials
of degree $2m$, and  such that
\[\int P_\alpha(\x) P_\beta(\x) \,d\mu(\x)\,=\,\int P_\alpha(\x) P_\beta(\x) \,Q(\x)\,d\mu(\x),\]
for all pairs $(P_\alpha, P_\beta)$ of orthonormal polynomials of degree $m$..

\section{Notation, definitions and preliminaries}
Let $\x=(x_1,\ldots,x_n)$ and let $\R[\x]$ be the ring of real polynomials in the variables $\x$
and $\R[\x]_d$ its subset of polynomials of degree at most $d$, which
is a vector space of dimension $s_d:={n+d\choose d}$.
The number of monomials of degree exactly $d$ is denoted by $r_d:={n+d-1\choose d}$.
A polynomial $f\in\R[\x]_d$ can be identified with its vector of coefficients
$(f_\alpha)=:\f\in\R^{s_d}$, in the canonical basis $(\x^\alpha)$, $\alpha\in\N^n$.
For any real symmetric matrix $\A$, the notation $\A\succeq0$ (resp. $\A\succ0$) stands for
$\A$ is positive semidefinite (resp. positive definite).

Let us define the {\it graded lexicographic order} (Glex) $\alpha\glexe\,\beta$ 
on $\N^n$, which first 
creates a partial order by $\vert\alpha\vert\leq\,\vert\beta\vert$, and then refines
this to a total order by breaking ties when $\vert\alpha\vert=\vert\beta\vert$ as would a dictionary with $x_1=a$, $x_2=b$, etc. For instance with $n=2$, the ordering reads
$1,x_1,x_2,x_1^2,x_1x_2,x_2^2, \ldots$.

For every integer
$m\in\N$, let $\x^m\in\R[\x]^{r_{m}}$ be the column vector of all monomials of degree $m$, 
with the $\glex$ ordering. For instance, with $n=2$ and $k=3$, $\x^3\in\R[\x]^4$ and $\x^3$ reads
\[\x^3\,=\,(x_1^3, x_1^2x_2,x_1x_2^2,x_2^3)^T.\]

With any sequence $\y=(y_\alpha)$, $\alpha\in\N^n$, one may associate a linear function
$\L_\y:\R[\x]\to \R$ by 
\begin{equation}
\label{riesz}
f\:(=\sum_\alpha f_\alpha\x^\alpha)\:\mapsto\quad \L_\y(f)\,=\,\sum_\alpha f_\alpha\,y_\alpha,\qquad f\in\R[\x].\end{equation}

By a slight abuse of notation, and for any $\v\in\R[\x]^p$ denote
by $\L_\y(\v)$ the vector $(\L_\y(v_j))$, $j\leq p$. Similarly,
for any matrix $\V\in\R[\x]^{p\times r}$ denote
by $\L_\y(\V)$ the matrix $(\L_\y(V_{ij}))$, $i\leq p$, $j\leq r$.
If $\mu$ is a finite and positive Borel measure with finite moments $\y=(y_\alpha)$, $\alpha\in\N^n$, then
\begin{equation}
\label{mom}
\L_\y(f)\,=\,\sum_{\alpha\in\N^n}f_\alpha y_\alpha\,=\,\sum_{\alpha\in\N^n}f_\alpha\int\x^\alpha\,d\mu(\x)
\,=\,\int fd\mu,\end{equation}
and $\mu$ is called a representing measure for $\y$.

\subsection*{Moment matrix}
The {\it moment} matrix $\M_d(\y)$ associated with a sequence $\y=(y_\alpha)$, is  
the real symmetric matrix with rows and columns indexed by $\N^n_d$, and whose entry $(\alpha,\beta)$ is just 
$\L_\y(\x^{\alpha+\beta})\,(=y_{\alpha+\beta})$, for every $\alpha,\beta\in\N^n_d$. Alternatively,
$\M_d(\y)\,=\,\L_\y ((1,\x,\ldots,\x^d)^T(1,\x,\ldots,\x^d))$.
Of course, if $\y$ has a representing measure $\mu$, then
$\M_d(\y)\succeq0$ for all $d\in\N$, because
\[\la\f,\M_d(\y)\f\ra\,=\,\L_\y(f^2)\,=\,\int f^2\,d\mu\,\geq0,\qquad\forall \,\f\,\in\R^{s_d}.\]
However, the converse is not true in general, i.e., $\M_d(\y)\succeq0$ for all $d$, does not imply that $\y$ has  representing measure.

\subsection{Orthonormal polynomials}
Most of the material of this section is taken from Helton et al. \cite{hlp}.
Given $\mu$ and $\y=(y_\alpha)$, $\alpha\in\N^n_{2d}$, as in (\ref{mom}), 
assume that $\M_d(\y)\succ0$ for every $d$.
Then one may define the scalar 
product $\la\cdot,\cdot\ra_\y$ on $\R[\x]_d$:
\[\la f,g\ra_\y\,:=\,\la\f,\M_d(\y)\g\ra\,=\,\int fg\,d\mu,\qquad \forall\,f,g\in\R[\x]_d.\]
With $d\in\N$ fixed, one may also associate a {\it unique} family of polynomials
$(P_\alpha)\subset\R[\x]_d$, $\alpha\in\N^n_d$, orthonormal with respect to $\mu$, as follows:
 \begin{equation}\label{defpolortho}
 \left\{
 \begin{array}{l}
 P_\alpha\in{\rm lin.span}\,\{\x^\beta\,:\:\beta \glexe\alpha\}\\
 \la P_\alpha,P_\beta\ra_\y\,=\,\delta_{\alpha=\beta},\quad\alpha,\beta\in\N^n_d\\
 \la P_\alpha,\x^\beta\ra_\y\,=\,0\quad\mbox{if }\beta \glex\alpha\\
 \la P_\alpha,\x^\alpha\ra_\y\,>\,0,\quad\alpha\in\N^n_{d},
 \end{array}\right.\end{equation}
 where $\delta_{\alpha=\beta}$ denotes the usual Kronecker symbol.
Existence and uniqueness of such a family is guaranteed by the Gram-Schmidt orhonormalization process following the `Glex" order of monomials, and by positivity of the moment matrix $\M_d(\y)$. See e.g. \cite[Theorem 3.1.11, p. 68]{dunkl}.

\subsection*{Computation}

Computing the family of orthonormal polynomials is relatively easy once the moment matrix
$\M_d(\y)$ is available. Suppose that one wants to compute $p_\sigma$ for some 
$\sigma\in\N^n_d$.

Build up the submatrix $\M^\sigma_d(\y)$ obtained from $\M_d(\y)$ by keeping only those 
columns indexed by $\beta\glexe\sigma$ and with
rows indexed by $\alpha\glex\sigma$. Hence $\M^\sigma_d(\y)$ has one row less than columns. Complete $\M^\sigma_d(\y)$ with an additional last row described by
$\M^\sigma_d(\y)[\sigma,\beta]=\x^\beta$, $\beta\glexe\sigma$. Then, up to a normalizing constant, $P_\sigma={\rm det}\,\M^\sigma_d(\y)$.  
For instance with $n=2$, $d=2$ and $\sigma=(11)$, one has
\begin{equation}
\label{compute}
\x\mapsto P_{(11)}(\x)\,=\,M\,{\rm det}\,\left[\begin{array}{ccccc}
y_{00}&y_{10}&y_{01}&y_{20}&y_{11}\\
y_{10}&y_{20}&y_{11}&y_{30}&y_{21}\\
y_{01}&y_{11}&y_{02}&y_{21}&y_{12}\\
y_{20}&y_{30}&y_{21}&y_{40}&y_{31}\\
1&x_1&x_2&x_1^2&x_1x_2\end{array}\right].\end{equation}
where the constant $M$ is chosen so that $\int P_{(11)}^2d\mu=1$; eee e.g. \cite{dunkl,hlp}.

\section{Gaussian cubature formula}

Given a finite Borel measure $\mu$ on $\K\subset\R^n$ with all moments finite,
points $\x(i)\in\R^n$, and weights $\gamma_i\in\R$, $i=1,\ldots s$, 
the linear functional $I:L_1(\mu)\to\R$,
\begin{equation}
\label{cuba}
f\mapsto I(f)\,:=\,\sum_{i=1}^s \gamma_i f(\x(i)),\quad f\in\mathcal{F},\end{equation}
is called a {\it cubature} formula of degree (or precision) $p$, if 
$I(f)=\int_\K f\,d\mu$ for all $f\in\R[\x]_p$; the points $(\x(i))$ are called the {\it nodes}. In other words, the approximation of the integral
$\int fd\mu$ by (\ref{cuba}) is exact for all polynomials of degree at most $p$.

Cubature formulas are the multivariate analogues of {\it quadrature} formulas in the one dimensional case.
By Tchakaloff's theorem, given $\mu$ on $\K$ with finite moments, and $d\in\N$, there exists
a measure $\nu$ finitely supported on at most $r_d$ points of $\K$, and whose moments up to order at least $d$
coincide with those of $\mu$; see e.g. Putinar \cite{putinar}. And so there exists
a cubature formula  of degree $d$. 
In the one-dimensional case, the celebrated {\it Gaussian quadrature} formula based on $d$ points (called nodes) of
$\K$ has precision $2d-1$ and positive weights. It exists as soon as the moment matrix of order $d$ is positive definite,
and all the nodes are zeros of the orthonormal polynomial of degree $d$. It is an important 
quadrature because among all quadrature formulas with same number of nodes, it is the one with maximum precision, 

When its multivariate analogue exists (then called a {\it Gaussian cubature}), the nodes are now the common zeros of 
all orthonormal polynomials of degree $d$, the precision is also $2d-1$, and the weights are also positive.
But its existence is not guaranteed even if all moment matrices are positive definite.
For a detailed account on cubatures and orhogonal polynomials, the interested reader is referred to the very nice book of Dunkl and Xu \cite{dunkl}, as well as the survey or Cools et al. \cite{cools}.

\subsection{Existence of Gaussian cubature}

Mysovskikh's theorem \cite[Theorem 3.7.4]{dunkl} 
states that a Gaussian cubature of degree $m$ exists if and only if the
orthonormal polynomials $\P_m$ have exactly ${n+m-1\choose n}$ common zeros.
And so checking existence reduces to computing all joint eigenvalues of $n$ so-called Jacobi matrices
of size $s_{m-1}$ described in e.g. \cite[p. 114]{dunkl}. We here provide an alternative 
criterion for existence which reduces to check whether a certain overdetermined linear system has a solution.

Let $\mu$ be a finite Borel measure with support $\supmu=\K\subset\R^n$,
and with all moments $\y=(y_\alpha)$, $\alpha\in\N^n$, finite.
We will adopt the same notation in \cite{dunkl}. Let $(P_\alpha)$, $\alpha\in\N^n$, be a system of orthonormal polynomials 
with respect to $\mu$,
and with each
$m\in\N$, let $\P_m=[P_{\vert\alpha\vert=m}]\in\R[\x]^{r_{m}}$ be the (column) vector 
$(P_{\alpha^{(m_1)}},\ldots,P_{\alpha^{(m_{r_m})}})^T$, where 
$(\alpha^{(m_i)})\subset\N^n_m$ are all monomials of degree $m$, with
$\alpha^{(m_1)}\glex\alpha^{(m_2)}\cdots \glex\alpha^{(m_{r_m})}$.
For instance with $n=2$,
\[\P_2\,=\,(P_{20},P_{11},P_{02})^T.\]

For any sequence $\z=(z_\alpha)$ denote by $\oM_d(\z)$ the moment matrix, with
rows and columns index in $\N^n_d$, and with entries
\[\oM_d(\z)[\alpha,\beta]\,=\,\L_\z(P_\alpha\,P_\beta),\qquad \alpha,\beta\in\N^n_d,\]
where $\L_\z:\R[\x]\to\R$ is the linear functional defined in (\ref{riesz}).
In other words, $\oM(\z)$ is the moment matrix expressed in the basis of orthonormal polynomials
and can be deduced from $\M_d(\z)$ by a linear transformation.
Notice that  if $\z=\y$ then $\oM_d(\y)$ is the identity matrix ${\rm{\bf I}}$.\\

Observe that for each $\gamma,\beta\in\N^n_m$, with $\vert\gamma\vert=\vert\beta\vert=m$,

\begin{equation}
\label{double1}
P_\gamma\,P_\beta\,=\,
\A^{(0,2m)}_{\gamma\beta}\,\P_0+\sum_{j=1}^{2m-1}\A^{(j,2m)}_{\gamma\beta}\,\P_j+
\A^{(2m,2m)}_{\gamma\beta}\,\P_{2m},\end{equation}
for some row vectors $\A^{(j,2m)}_{\gamma\beta}\in\R^{r_{j}}$, $j=0,1,\ldots,2m$.
Let $t_m:=r_{m}(r_{m}+1)/2$.

\begin{thm}
\label{thmain}
Let $\A^{(0,2m)}$ be the vector $(\A^{(0,2m)}_{\gamma\beta})\in\R^{t_m}$, and let
$\A^{(2m,2m)}\in\R^{t_m\times r_{2m}}$ be the matrix whose
rows are the vectors $\A^{(2m,2m)}_{\gamma\beta}$, defined in (\ref{double1})
for $\gamma,\beta\in\N^n_m$ with $\vert\gamma\vert=\vert\beta\vert=m$.

There exists a Gaussian cubature formula of degree $2m-1$ if and only if the linear system
\begin{equation}
\label{thmain-1}
\A^{(0,m)}+\A^{(2m,m)}\,\u\,=\,0,\end{equation}
has a solution $\u\in\R^{r_{2m}}$.
\end{thm}
\begin{proof}
{\it Only if part.} If there exists a Gaussian cubature (hence of degree $2m-1$),
there is a finite Borel measure $\nu$ supported on the $s_{m-1}$ common zeros of $\P_m$; see \cite[Theorem 3.7.4]{dunkl}.
Let $\z=(z_\alpha)$, $\alpha\in\N^n_{2m}$, be the moments up to order $2m$, of the Borel measure $\nu$.
Since the cubature  is exact for polynomials in $\R[\x]_{2m-1}$, 
\[z_\alpha\,=\,\L_\z(\x^\alpha)\,=\,\L_\y(\x^\alpha)\,=\,y_\alpha,\qquad \forall \alpha\in\N^n_{2m-1}.\]
In addition, $\nu$ being supported on the common zeros of $\P_m$, we also have
\[\L_\z(P_\alpha^2)=\int P_\alpha(\x)^2\nu(d\x)\,=\,0,\qquad\forall \alpha\in\N^n_m\mbox{ with }\vert\alpha\vert=m,\]
that is, all diagonal elements of the positive semidefinite matrix 
$\L_\z(\P_m\P_m^T)$ vanish, which in turn implies $\L_\z(\P_m\P_m^T)=0$.
Combining with (\ref{double1}) yields
\begin{eqnarray*}
0\,=\,\L_\z(P_\gamma P_\beta)&=&
\A^{(0,2m)}_{\gamma\beta} +\sum_{j=1}^{2m-1}\A^{(j,2m)}_{\gamma\beta} \underbrace{\L_\z(\P_j)}_{=\L_\y(\P_j)=0}
+\A^{(2m,2m)}_{\gamma\beta} \L_\z(\P_{2m})\\
&=&\A^{(0,2m)}_{\gamma\beta} 
+\A^{(2m,2m)}_{\gamma\beta} \L_\z(\P_{2m})\\
&=&\left(\A^{(0,2m)}+\A^{(2m,2m)}\,\L_\z(\P_{2m})\right)(\gamma,\beta)\end{eqnarray*}
which is (\ref{thmain-1}) with $\u=\L_\z(\P_{2m})$.

{\it If part.} Conversely, assume that (\ref{thmain-1}) holds for some vector $\u\in\R^{r_{2m}}$.
Consider the sequence $\z=(z_\alpha)$, $\alpha\in\N^n_{2m}$,
with $z_\alpha=y_\alpha$ for all $\alpha\in\N^n_{2m-1}$ and
$\L_\z(P_\alpha)=u_\alpha$ for all $\alpha\in\N^n_{2m}$, with $\vert\alpha\vert=2m$.
(Equivalently, $\L_\z(\P_{2m})=\u$.)
And indeed such a sequence exists.
It suffices to determine $z_\alpha$ for all $\alpha\in\N^n_{2m}$ with $\vert\alpha\vert=2m$
(equivalently, $\L_\z(\x^{2m})$).
Recall the notation 
$\x^k=(\x^\alpha)$, $\alpha\in\N^n$, with $\vert\alpha\vert=k$,
and recall that for each $d\in\N$, $(1,\P_1,\ldots,\P_d)$ is a basis of $\R[\x]_d$, and so,
\begin{equation}
\label{triangular}
\left[\begin{array}{l}1\\\P_1\\ \cdot\\ \cdot\\ \P_{2m}\end{array}\right]\,=\,
\underbrace{\left[\begin{array}{ccccc} 1&0&\cdot&\cdot&0\\
\bS_{01}&\bS_1&0&\cdot &0\\
\cdot&\cdot&\cdot&\cdot&\cdot\\
\cdot&\cdot&\cdot&\cdot&\cdot\\
\bS_{02m}&\bS_{12m}&\cdot &\bS_{(2m-1)2m}&\bS_{2m}\end{array}\right]}_{=\bS}\, 
\left[\begin{array}{l}1\\
\x^1\\ \cdot\\ \cdot\\  \x^{2m}\end{array}\right]\end{equation}
for some change of basis matrix  $\bS$ which is invertible and block-lower triangular.
And so $\z$ solves 
\[\u\,=\,\bS_{2m} \,\L_\z(\x^{2m})+\Theta\underbrace{\L_\z(1,\x^1,\cdots,\x^{2m-1})^T}_{=\y_{2m-1}},\]
where $\Theta\in\R^{r_{2m}\times s(2m-1)}$ is the matrix $[\bS_{02m}\vert\,\cdots\vert\bS_{(2m-1)2m}]$,
and $\y_{2m-1}=(y_\alpha)$, $\vert\alpha\vert\leq 2m-1$.
And so $\L_\z(\x^{2m})=\bS_{2m}^{-1}\u-\bS_{2m}^{-1}\Theta\,\y_{2m-1}$.

Next, the moment matrix $\oM_m(\z)$, which by definition reads
\[\left[\begin{array}{ccc}\oM_{m-1}(\z)& \vert &\begin{array}{c}\L_\z(\P_m^T)\\ \L_\z(\P_1\P_m^T)
\\\cdot\\ \cdot\\\L_\z(\P_{m-1}\P_m^T)\end{array}\\
-&&-\\
\begin{array}{lllll}\L_\z(\P_m)&\L_\z(\P_m\P_1)^T&\cdots&\cdots \L_\z(\P_m\P_{m-1}^T)\end{array}&\vert &\L_\z(\P_m\P_m^T)\end{array}\right],\]
simplifies to
\begin{equation}
\label{decompp}
\oM_m(\z)\,=\,\left[\begin{array}{ccc}{\rm{\bf I}}& \vert &0\\
-&&-\\
0&\vert &\L_\z(\P_m\P_m^T)\end{array}\right],
\end{equation}
because $\L_\z(\P_i\P_j^T)=\L_\y(\P_i\P_j^T)=0$ for all $i\neq j$ with $i+j\leq 2m-1$.
But then we also have $\L_\z(\P_m\P_m^T)=0$ because from (\ref{double1}),
\begin{eqnarray*}
\L_\z(P_\gamma P_\beta)&=&
\A^{(0,2m)}_{\gamma\beta} +\sum_{j=1}^{2m-1}\A^{(j,2m)}_{\gamma\beta} \underbrace{\L_\z(\P_j)}_{=\L_\y(\P_j)=0}
+\A^{(2m,2m)}_{\gamma\beta} \underbrace{\L_\z(\P_{2m})}_{=\u}\\
&=&\left(\A^{(0,2m)}+\A^{(2m,2m)}\,\u\right)(\gamma,\beta)=0,\end{eqnarray*}
for all $\gamma,\beta\in\N^n_m$ with $\vert\gamma\vert=\vert\beta\vert=m$.

And so $\oM_{m}(\z)\succeq0$, and ${\rm rank}\,\oM_m(\z)={\rm rank}\,\oM_{m-1}(\z)$.
By the flat extension theorem of Curto and Fialkow, 
$\z$ is the moment sequence of
a finite Borel measure $\psi$ on $\R^n$, supported on ${\rm rank}\,\oM_m(\z)=s_{m-1}$ distinct points
$\x(k)$, $k=1,\ldots,s_{m-1}$; see e.g. Curto and Fialkow \cite{curto1,curto2}, Lasserre \cite[Theorem 3.7]{lassbook}, or Laurent \cite{laurent}.
That is, denoting by $\delta_{\x}$ the Dirac measure at $\x\in\R^n$,
\[\psi\,=\,\sum_{k=1}^{s_{m-1}}\gamma_k\,\delta_{\x(k)},\]
for some strictly positive weights $\gamma_k$.
From $0=\L_\z(\P_m\P_m^T)$ we obtain
\[0\,=\,\L_\z(P_\alpha^2)\,=\,\int_{\R^n}P_\alpha(\x)^2\,d\psi(\x),\quad\forall \alpha\in\N^n_{m},\mbox{ with }\vert\alpha\vert=m,\]
which proves that each $\x(k)$ is a common zero of $\P_m$, $k=1,\ldots,s_{m-1}$.
But by \cite[Corollary 3.6.4]{dunkl}, $\P_m$ has at most $s_{m-1}$ common zeros. Therefore,
$\P_m$ has exactly $s_{m-1}$ common zeros and $\psi$ is supported on {\it all} common zeros of $\P_m$.
By \cite[Theorem 3.7.4]{dunkl}, $\mu$ admits a Gaussian cubature formula of degree $2m-1$.
Indeed, since $z_\alpha=y_\alpha$ for all $\vert\alpha\vert\leq 2m-1$, then for every $f\in\R[\x]_{2m-1}$,
\[\int f(\x)\,d\mu(\x)\,=\,\L_\y(f)\,=\,\L_\z(f)\,=\,\int fd\psi\,=\,\sum_{k=1}^{s_{m-1}}\gamma_k\,f(\x(k)),\]
with $\gamma_k>0$ for every $k$.
\end{proof}
By \cite[Theorem 7]{cools}, we also know that the $\x(k)$'s belong to the interior of the convex hull ${\rm conv}\,\K$ of $\K$.

Observe that existence of a solution $\u\in\R^{r_{2m}}$ for the linear system (\ref{thmain-1}) 
imposes drastic conditions on the linear functional
$\L_\y$ because (\ref{thmain-1}) is always overdetermined. Moreover, for fixed $m$ (resp. for fixed $n$)
the larger $n$ (resp. the larger $m$) the worse it gets. This is because
$t_m\,(=r_{m}(r_{m}+1)/2)$ increases faster than $r_{2m}$.

\subsection*{Interpretation} We next provide an explicit expression of the vector $\A^{(0,2m)}$
as well as the matrix $\A^{(2m,2m)}$, which permits to
obtain an interpretation of the condition (\ref{thmain-1}).

For $\gamma,\beta\in\N^m$, applying the linear functional $\L_\y$ to both sides of
(\ref{double1}) yields
\begin{equation}
\label{constant}
\delta_{\gamma=\beta}\,=\,\L_\y(P_\gamma P_\beta)\,=\,\A^{(0,2m)}_{\gamma\beta},\qquad \forall \vert\gamma\vert=\vert\beta\vert=m.\end{equation}
because $\L_\y(\P_k)=0$ for all $k$.  This provides an explicit expression of the vector $\A^{(0,2m)}$.
Similarly, multiplying both sides of (\ref{double1}) by $P_\kappa$ with $\vert\kappa\vert=2m$, and applying again $\L_\y$, yields
\begin{equation}
\label{term2m}
\L_\y(P_\gamma P_\beta P_\kappa)\,=\,\sum_{\vert\alpha\vert=2m}
\A^{(2m,2m)}_{\gamma\beta}(\alpha)\,\L_\y(P_\alpha P_\kappa)\,=\,\A^{(2m,2m)}_{\gamma\beta}(\kappa),\end{equation}
because $\L_\y(P_\alpha P_\kappa)=\delta_{\alpha=\kappa}$.
 This provides an explicit expression for all the entries of the matrix $\A^{(2m,2m)}$. And we obtain:
 
 \begin{cor}
 \label{cor-final}
 There exists a Gaussian cubature of degree $m$ if and only if there exists a polynomial $Q=\u^T\P_{2m}$,
 for some $\u\in\R^{r_{2m}}$, such that
 \begin{equation}
 \label{cor-final1}
 \L_\y(P_\gamma P_\beta)\,=\,\L_\y(P_\gamma P_\beta\,Q),
 \end{equation}
 for all $\beta,\gamma\in\N^n_m$ with $\vert\gamma\vert,\,\vert\beta\vert=m$.
 Equivalently, (\ref{cor-final1}) reads
 \[\int P_\gamma(\x) P_\beta(\x)\,d\mu(\x)\,=\,\int P_\gamma(\x) P_\beta(\x)\,Q(\x)\,d\mu(\x),\]
 for all $\beta,\gamma\in\N^n_m$ with $\vert\gamma\vert,\,\vert\beta\vert=m$, or in matrix form:
 \[{\rm{\bf I}}\,=\,\int \P_m(\x)\P_m(\x)^T\,Q(\x)\,d\mu(\x),\]
 \end{cor}
 \begin{proof}
 With $\u$ as in (\ref{thmain-1}) and using (\ref{constant})-(\ref{term2m}), for all $\beta,\gamma\in\N^n_m$ with $\vert\gamma\vert,\,\vert\beta\vert=m$, one obtains
 \begin{eqnarray*}
\L_\y(P_\gamma P_\beta)\,=\,\A^{(0,2m)}_{\gamma\beta}&=& (\A^{(2m,2m)}\u)_{\gamma,\beta}\\
&=&\L_\y\left(P_\gamma P_\beta \,(\sum_{\alpha\in\N^n_{2m},\vert\alpha\vert=2m}u_\alpha P_\alpha)\right)\\
&=&\L_\y\left(P_\gamma P_\beta \,(\u^T\P_{2m})\right),
\end{eqnarray*}
which  is   (\ref{cor-final1}) with $Q=\u^T\P_{2m}$. Conversely, if (\ref{cor-final1}) holds with
  for some $Q=\u^T\P_{2m}$ then obviously $\u$ is a solution of (\ref{thmain-1}).
 \end{proof}
 
\subsection*{Final remark.}
One may obtain more information about the polynomial $Q$ of Corollary \ref{cor-final}.
From the proof of Theorem \ref{thmain}, for any solution $\u\in\R^{r_{2m}}$ of (\ref{thmain-1}), one has
$\u=\L_\z(\P_{2m})$ where $\z$ is the moment vector of the measure supported on
the $s_{m-1}$ nodes of the Gaussian cubature. That is,
\[\u\,=\,\L_\z(\P_{2m})\,=\,\sum_{k=1}^{s_{m-1}}\gamma_k\,\P_{2m}(\x(k)),\] 
for some positive weights $\gamma_k$.
On the other hand, the polynomial $Q$ of Corollary \ref{cor-final} is of the form
$\u^T\P_{2m}$ where $\u$ solves (\ref{thmain-1}). Therefore,
\[\x\mapsto Q(\x)=\u^T\P_{2m}(\x)\,=\,\sum_{k=1}^{s_{m-1}}\gamma_k\,\P_{2m}(\x(k))^T\P_{2m}(\x).\]
Hence $Q$ satisfies
\[\L_\y(P_\alpha \,Q)\,=\,\u^T\L_\y(P_\alpha\,\P_{2m})\,=\,0,\quad\forall \vert\alpha\vert < 2m\]
(in particular $\displaystyle\int Qd\mu=0$), and
\[\L_\y(P_\alpha \,Q)\,=\,\u_\alpha\,=\,\sum_{k=1}^{s_{m-1}}\gamma_k P_\alpha(\x(k)),\quad\forall \vert\alpha\vert = 2m.\]


\begin{thebibliography}{las}
\bibitem{cools}
R. Cools, I.P. Mysovskikh, H.J. Schmid. Cubature formulae and orthogonal polynomials,
J. Comp. Appl. Math. {\bf 127} (2001), pp. 121--152.
\bibitem{curto1}
R. Curto and L. Fialkow. Flat extensions of positive moment matrices: recursively generated relations, Memoirs. Amer. Math. Soc.
{\bf 136}, AMS, Providence, 1998.
\bibitem{curto2}
R. Curto and L. Fialkow. The truncated $\K$-moment problem,
Trans. Amer. Math. Soc. {\bf 352}, pp. 2825--2855.
{\bf 136}, AMS, 1998.
\bibitem{dunkl}
C.F. Dunkl  and Y. Xu. {\it Orthogonal Polynomials of Several Variables}, Cambridge University Press, Cambridge, 2001.
\bibitem{hlp}
J.W. Helton, J.B. Lasserre and M. Putinar. Measures with zeros in the inverse of their moment matrix, Annals of Prob. {\bf 36} (2008), pp. 1453--1471.
\bibitem{laurent}
M. Laurent. Revisiting two theorems of Curto and Fialkow on moment matrices,
Proc. Amer. Math. {\bf 133} (2005), pp. 2965--2976.
\bibitem{lassbook}
J.B. Lasserre. {\it Moments, Positive Polynomials and Their Applications},
Imperial College Press, London, 2009.
\bibitem{putinar}
M. Putinar. {\it A note on Tchakaloff's theorem}, Proc. Amer. Math. Soc. {\bf 125} (1997), pp. 2409?2414. 
\end{thebibliography}
\end{document}